\documentclass[11pt, oneside]{amsart}   	
\usepackage{geometry}                		
\usepackage{graphicx}				
								
\usepackage{amssymb}
\newtheorem{theorem}{Theorem}
\newtheorem{lemma}{Lemma}

\def\No{{\mathcal N}}
\def\E{{\mathsf{E}}}
\def\RR{{\mathbb R}}
\def\pairs{\operatorname{PP}}
\def\Ker{\operatorname{Ker}}
\def\Vect{\operatorname{Vect}}
\def\vv{{\bf v}}
\def\a{{\bf a}}
\def\0{{\bf 0}}
\def\Y{{\bf Y}}
\def\X{{\bf X}}
\def\x{{\bf x}}
\def\bsigma{{\bf a}}

\title{A short proof of Isserlis' theorem}

\author{Hans Z. Munthe-Kaas}
\address{Department of Mathematics and Statistics, UiT – The Arctic University of Norway, Tromsø, Norway. Department of Mathematics, University of Bergen, Bergen, Norway. }
\email{Hans.Munthe-Kaas@uib.no}

\author{Olivier Verdier}
\address{Department of Computing, Electrical Engineering and Mathematical Sciences,Western Norway University of Applied Sciences, Bergen, Norway.}
\email{olivier.verdier@gmail.com}

\author{Gilles Vilmart}
\address{Section of Mathematics, University of Geneva.}
\email{Gilles.Vilmart@unige.ch}

\date{\today}							

\begin{document}
\begin{abstract}
We show that Isserlis' theorem follows as a corollary to the invariant tensor theorem for isotropic tensors. 
\end{abstract}
\maketitle

\section{Isserlis's theorem}
Isserlis' theorem was first derived by Leon Isserlis in 1918~\cite{isserlis}. 
This result, also known as Wick's probability theorem, permits to express arbitrary high-order moments of a given multivariate normal distribution $\Y$ in terms of the coefficients of its covariance matrix $\Sigma$.
\begin{theorem}[Isserlis] \label{thm:isserlis} Let $\Y= (Y_1,Y_2,\ldots,Y_n)^T \sim \No_{n}({\bf 0},\Sigma)$ be a multivariate normally distributed vector in $\RR^n$ with mean zero.
Then 
\begin{equation} \label{eq:isserlis}
\E(Y_1 Y_2\cdots Y_n) = 
\sum_{p\in \pairs(n)}\prod_{(\ell,r)\in p} \E(Y_\ell Y_r).
\end{equation}
\end{theorem}
The sum is over all partitions of $n$ into disjoint pairs, denoted $\pairs(n)$. 
For odd $n$, $\pairs(n)=\emptyset$ is empty and \eqref{eq:isserlis} reduces to $E(Y_1 Y_2\cdots Y_n)=0$, and we recover that odd moments of normal distributions with mean zero always vanish.
For even $n$ with $n=2k$, the number of elements of $\pairs(n)$ is exactly
$
\# \pairs(n) = \frac{(2k)!}{2^k k!}.
$
For $n=4$, we have $\# \pairs(4)=3$ with
\[\pairs(4) = \{\{(1,2),(3,4)\},\{(1,3),(2,4)\},\{(1,4),(2,3)\}\},\] and formula \eqref{eq:isserlis} becomes
\[\E(Y_1 Y_2 Y_3 Y_4)= \E(Y_1 Y_2) \E(Y_3 Y_4) + \E(Y_1 Y_3) \E(Y_2 Y_4)+ \E(Y_1 Y_4) \E(Y_2 Y_3). \]
This result was shown by Isserlis first for $n=4$ in \cite{isserlis_1916} before he generalized the result for arbitrary integer~$n$ in \cite{isserlis} with an elementary proof of \eqref{eq:isserlis} by induction. 

\section{Isserlis' Theorem as a corollary of the invariant tensor theorem}
Let $V, (\_,\_)$ be a Euclidean space (e.g.\ $\RR^d$ with the standard inner product). A tensor $T\colon \otimes^n V\rightarrow \RR$, where we denote $\otimes^n V=V\otimes V\otimes \cdots \otimes V$ ($n$ times), is called \emph{isotropic}
if $T(Q\vv_1,Q\vv_2,\ldots,Q\vv_n) = T(\vv_1,\ldots, \vv_n)$ for all orthogonal matrices $Q\in O(V)$ and all vectors $\vv_i\in V$. 
Note that the scalar product $(\_,\_)$ itself is an example of an isotropic tensor $V\otimes V\rightarrow \RR$.

We shall derive formula \eqref{eq:isserlis} as a consequence of the invariant tensor theorem for isotropic tensors~\cite[Sec.\ II.9]{weyl},~\cite[Appendix I]{ABP73},~\cite[\S 33.2]{kolar} stating that the linear space of isotropic tensors is spanned by certain elementary isotropic tensors with two inputs. 

\begin{theorem}[Invariant tensor theorem for isotropic tensors] An isotropic tensor $T$ decomposes in a sum of products of pairwise inner products:
\begin{equation}T(\vv_1, \vv_2, \ldots,\vv_n) = 
\sum_{p\in \pairs(n)} \alpha_p\prod_{(\ell,r)\in p} (\vv_\ell, \vv_r)
\label{eq:T}
\end{equation}
for some scalars $\alpha_p\in \RR$ independent of $\vv_1, \vv_2, \ldots,\vv_n$. If in addition, $T$ is symmetric, i.e. invariant under permutations of its inputs, then all $\alpha_p$ are equal, i.e. independent of $p$. 
\end{theorem}
\begin{proof}
The first part is a reformulation of \cite[Th.~5.3.5]{GoodmanWallach2009}. The second part is obtained by contradiction.
\end{proof}
Again, for odd integers $n$, $\pairs(n)=\emptyset$ and the right hand side in \eqref{eq:T} vanishes.

We call a random variable $\X$ in $\RR^d$ \emph{isotropic} if its distribution is invariant by rotation, i.e.,
if for any $Q \in O(d)$ and any integrable function $f\colon\RR^d\mapsto\RR$ we have
$\E (f(\X)) = \E(f(Q \X))$.
For such a random variable, and all integer $k$, we define the quantities:
\[
  c_{k}(\X) := \frac{2^k k!}{(2k)!} \frac{\E(X^{2k})}{\E(X^2)^k}
  ,
\]
where $X := \a^T \X$ for an arbitrary, nonzero, deterministic vector $\a \in \RR^d$ (note that $c_{k}(\X)$ does not depend on the choice of $\a$ by the isotropy of $\X$).
Note also that $c_1 = 1$.

We recall the following classical result on the covariance matrix of isotropic random variables. 
\begin{lemma}
\label{prop:isocov}
  Let $\X$ be  an isotropic random variable in $\RR^d$.
  Then $\E(\X \X^T) = \lambda I_d$ for some $\lambda \in \RR$.
\end{lemma}
\begin{proof}
The definition of isotropy with $f(\x)=\x\x^T$ yields that the matrix $\Sigma=\E(\X \X^T)$ commutes with all orthogonal matrices $Q\in O(d)$ a property satisfied only by matrices of the form $\lambda I_d$.
Indeed, if a matrix $\Sigma$ commutes in particular with all orthogonal symmetries~$S$, using that $\Sigma$ leaves invariant the eigenspace $\Ker(S+I_d)$ of $S$, we obtain $\Sigma \x\in \Vect(\x)$ for all $\x\in \RR^d$. This later property yields $\Sigma \x=\lambda \x$ for all $\x\in \RR^d$ and for some $\lambda \in \RR$ which is shown to be independent of $\x$ by contradiction and this concludes the proof.
\end{proof}

\begin{theorem}
  Let $\X$ be  an isotropic random variable in $\RR^d$.
  Then for any choice of $n$ deterministic vectors $\a_1$,\ldots,$\a_n\in \RR^d$, we have
  \[
    \E(\a_1^T \X \cdots \a_n^T \X) = 
\begin{cases}
c_k(\X)\sum_{p\in \pairs(n)}\prod_{(\ell,r)\in p} \E(\a_\ell^T \X \, \a_r^T \X), & \text{if $n = 2k$,}\\
0 & \text{if $n$ is odd.}
\end{cases}
  \]
\end{theorem}

\begin{proof}
  Let $A= (\a_1, \a_2,\ldots,\a_n)$ be the matrix containing the vectors $\a_i$ in its columns, and define $Y := A^T \X$.
  As $X$ is isotropic, the tensor $T(\a_1,\a_2,\ldots,\a_n) := \E(Y_1 Y_2\cdots Y_n) = \E(\a_1^T\X \, \a_2^T\X\cdots \a_n^T\X)$ is an isotropic and symmetric tensor,
  and hence of the form~\eqref{eq:T} for  some real constant $\alpha_p = c$.
  We now compute $\E(Y_\ell Y_r) = \E(\a_\ell^T \X \, \a_r^T\X) = \E(\a_{\ell}^T \X \X^T \a_r) = \a_{\ell}^T \E(\X \X^T) \a_r=\lambda(\a_\ell, \a_r)$ where we used Lemma~\ref{prop:isocov}.
  Finally, when $n = 2k$, choosing $\a_j = \a$, for an arbitrary vector $\a$ shows that $c = c_{k}(\X)$.
\end{proof}

Using that the moment generating function of the standard normal distribution is $\mathrm{e}^{t^2/2}$, one can show that $\E(X^{2k}) = \frac{(2k)!}{2^k k!} = \# \pairs(2k)$,
from which we deduce for the case of multivariate normally
distributed vectors with mean zero, $c_k(\X) = 1$.
This permits to deduce the proof of Isserlis' Theorem \ref{thm:isserlis}.


\section{Perspectives}

Let us mention that the Isserlis theorem plays a fundamental role in the construction of exotic aromatic Butcher series, which are formal series index by certain graphs, introduced in \cite{laurent_vilmart_2020,laurent_vilmart_2022} for the numerical analysis of numerical integrators for ergodic stochastic dynamics. 
Such exotic aromatic Butcher series also yields new algebraic structures presented recently in \cite{Bronasco22,bronasco_laurent_2024}, generalizing aromatic Butcher series \cite{bogfjellmo_19,Laurent_McLachlan_Munthe-Kaas_Verdier_2023} with a new kind of edge denoted \textit{lianas} \cite{laurent_vilmart_2020} and paring any two nodes, in the spirit of the notion of paring in Isserlis' theorem. 
The construction of such lianas is based on the following result valid for an arbitrary tensor $T$ (without any isotropy or symmetry assumption on $T$), which can be proved as a straightforward consequence of the Isserlis theorem. 
\begin{theorem} (see \cite[Theorem 4.1]{laurent_vilmart_2020})
Consider a normally distributed vector
  $\Y \sim \No_{d}({\bf 0},\Sigma)$ in $\RR^d$ with mean zero and covariance matrix $\Sigma=AA^T$ with $A=(\bsigma_1,\ldots,\bsigma_d)$.
 Then, for all tensor $T\colon \otimes^n \RR^d\rightarrow \RR$, we have
  \[
    \E\big(T(\Y,\Y, \ldots, \Y)\big) = 
\sum_{p\in \pairs(n)}
\sum_{(k_1,k_2,\ldots,k_n)\in K_d(p)}
T(\bsigma_{k_1},\ldots, \bsigma_{k_n})
  \]
where for each pairing $p\in PP(n)$, we define the set of paired indices 
  \[
  K_d(p) := \{(k_1,\ldots,k_n) \in \{1,\ldots, d\}^n\ ;\ k_\ell=k_r \mbox{ for all } (\ell,r) \in p\}.
\]
\end{theorem}
Note that for all $p\in \pairs(2k)$, the set $K_d(p)$ has $d^{k}$ elements. For instance, for $n=4$,
 \[
\E\big(T(\Y,\Y,\Y,\Y)\big) = \sum_{i,j=1}^d \big(T(\bsigma_{i},\bsigma_{i},\bsigma_{j},\bsigma_{j})
+T(\bsigma_{i},\bsigma_{j},\bsigma_{i},\bsigma_{j})
+T(\bsigma_{i},\bsigma_{j},\bsigma_{j},\bsigma_{i})\big).
\]
We believe that this short proof of Isserlis' theorem using the invariant tensor theorem for isotropic tensors could give new insight on further studies of equivariance characterization properties for Butcher series \cite{McLachlan_Modin_MuntheKaas_Verdier_2016} and generalisations~\cite{laurent_munthekaas_2024}.

\medskip
\noindent
\textbf{Acknowledgements:} We thank Arieh Iserles for enlightening us on the genealogy of Moses~Isserles. 
H. M.-K. is supported by the Research Council of Norway through
project 302831 Computational Dynamics and Stochastics on Manifolds (CODYSMA).
G.V. was partially supported by the Swiss National Science Foundation, projects No 200020\_214819 and No. 200020\_192129. 

\bibliographystyle{amsplain}
\providecommand{\bysame}{\leavevmode\hbox to3em{\hrulefill}\thinspace}
\providecommand{\MR}{\relax\ifhmode\unskip\space\fi MR }
\providecommand{\MRhref}[2]{%
  \href{http://www.ams.org/mathscinet-getitem?mr=#1}{#2}
}
\providecommand{\href}[2]{#2}

\end{document}